\documentclass[12pt]{article}
\usepackage{latexsym,amssymb,amsmath,amsfonts,amsthm, enumerate}
\usepackage{comment}
\newtheorem{thm}{Theorem}

\newtheorem{lemma}[thm]{Lemma}

\newtheorem{cor}[thm]{Corollary}

\def\beq{ \begin{equation} }
\def\eeq{ \end{equation} }

\def\ep{\epsilon}

\def\square{\vcenter{\vbox{\hrule height .4pt
  \hbox{\vrule width .4pt height 5pt \kern 5pt
        \vrule width .4pt} \hrule height .4pt}}}
\def\eopt{\hfill$\square$}

\def\ZZ{\mathbb{Z}}
\def\CC{\mathbb{C}}
\def\LL{\mathbb{L}}

\def\var{\hbox{var}\,}

\def\rad{\hbox{rad}}
\def\mathds#1{#1}
\newcommand{\dual}{\mathbb{L}}

\title{Poisson percolation on the square lattice}
\author{Irina Cristali, Matthew Junge and Rick Durrett, }

\usepackage{graphicx}
\graphicspath{%
    {converted_graphics/}
    {/}
}
\begin{document}

\maketitle 

\begin{abstract}
 On the square lattice raindrops fall on an edge with midpoint $x$ at rate $\|x\|_\infty^{-\alpha}$. The edge becomes open when the first drop falls on it. Let $\rho(x,t)$ be the probability that the edge with midpoint $x=(x_1,x_2)$ is open at time $t$ and  let $n(p,t)$ be the distance at which edges are open with probability $p$ at time $t$. We show that with probability tending to 1 as $t \to \infty$: (i) the cluster containing the origin $\CC_0(t)$ is contained in the square of radius $n(p_c-\ep,t)$, and (ii) the cluster fills the square of radius $n(p_c+\ep,t)$ with the density of points near $x$ being close to $\theta(\rho(x,t))$ where $\theta(p)$ is the percolation probability when bonds are open with probability $p$ on $\ZZ^2$. Results of Nolin suggest that if $N=n(p_c,t)$ then the boundary fluctuations of $\CC_0(t)$ are of size $N^{4/7}$. 
\end{abstract}

\section{Introduction} \label{Intro}

We study the geometry of the open cluster containing the origin in a nonhomogeneous version of bond percolation on the two-dimensional square lattice that we call \emph{Poisson percolation}. On the lattice an edge with midpoint $x$ is assigned an independent Poisson processes with rate $\|x\|^{-\alpha}$ where $\|x\| = \max\{|x_1|,|x_2|\}$ is the $L^\infty$ norm. The edge becomes open at the time of the first arrival. Our inspiration comes from the \emph{rainstick process}. It was introduced by Pitman and Tang in \cite{PT17} with followup work by Pitman, Tang and Duchamps \cite{DPT}. In this discrete time process raindrops fall one after the other on the positive integers and sites become wet when landed on. The locations of raindrop landings are independent random variables with a geometric distribution: $P(X=k) = (1-p)^{k-1}p$ for $k \ge 1$. Let $T$ be the first time that the configuration is a single wet component containing 1, and let $K$ be its length. Pitman and Tang observed in \cite{PT17} that the value of $K$ describes the size of the first block in a family of \emph{regenerative permutations}. Understanding block sizes has been useful for studying the structure of random Mallows permutations \cite{BB,GP16}. 

In \cite{CRS} the asymptotic behavior of $T$ and $K$ as $p \to 0$ was studied. They proved that $T \approx \exp(e^{c/p})$ and $K \approx e^{c/p}$, where $c \approx 1.1524$ is a constant defined by an integral. This says that the first block is large, and takes a very large amount of time to form. 
It turns out that an exponentially decaying tail is needed for the rainstick process to terminate with probability one. Theorem 5 in \cite{CRS} shows that if raindrops land beyond site $k$ with probability $\exp(-k^\beta)$ for $\beta <1$ then $T$ is infinite with positive probability. 

The Poisson percolation we study here is a higher dimensional version of the rainstick process. In both processes distant edges are less likely to become open (wet). However, we have a power-law tail rather than a geometric, so it is likely that there is no time at which there is a single component. So, we will instead study the size and density of the wet cluster containing the origin.

To state our results we introduce some notation. Here, we study Poisson percolation only on the two dimensional lattice $\ZZ^2$. An edge with midpoint $x$ will be open at time $t$ with probability $\rho(x,t) = 1 - \exp(-t\|x\|^{-\alpha})$. We define the cluster containing the origin at time $t$ to be the set of points $\CC_0(t)$ that can be reached from the origin by a path of open edges. Let 
\beq
c_{p,\alpha} = (- \log (1-p))^{-1/\alpha}.
\label{cpa}
\eeq
A little algebra gives
$n(p,t) =  \max\{ \|x \|\colon \rho(x,t) \geq p \} = c_{p,\alpha} t^{1/\alpha}$.

Let $R(0,r) = \{ x \colon \| x \| \le r \}$ be the square with radius $r$ centered at 0. Recall that $p_c =1/2$ is the critical value for bond percolation on the two-dimensional lattice. For this and other facts we use about percolation, see Grimmett's book \cite{Gri97}. Our first result gives an upper bound on $\CC_0(t)$.

\begin{thm} \label{thm:outside}
For any $\ep>0$ the probability $\CC_0(t) \subseteq R(0,n(p_c-\ep,t))$ tends to 1 as $n \to\infty$. 
\end{thm}

Having shown that $\CC_0(t)$ is with high probability contained within $R(0,n(p_c-\ep,t))$, we would like to describe what it looks like inside $R(0,n(p_c+\ep,t))$. To do this we relate it to standard bond percolation on $\ZZ^2$. Let $\mathcal C_0$ be the open cluster containing the origin in bond percolation where each edge is open with probability $p$, and set $\theta(p) = P_p( |{\cal C}_0| = \infty)$, where $P_p$ is the probability measure for bond percolation on $\ZZ^2$, when edges are open with probability $p$. Intuitively, near $x \in R(0,n(p_c+\ep,t))$ the density of points in $\CC_0(t)$ will be close to $\theta(\rho(x,t))$.
 To state this precisely, let $n = n(p_c+\epsilon,t)$. Fix $1/2 < a < 1$ and tile the plane with boxes of side length $n^{a}$: 
 $$
R_{i,j} = [in^a,(i+1)n^a] \times [jn^a,(j+1)n^a],
$$
with center $x_{i,j}$. Let $D_{i,j} = |\CC_0(t) \cap R_{i,j}|/n^{2a}$ be the density of points in $R_{i,j}$ that belong to ${\CC}_0(t)$ and let $\Lambda(t,\ep) = \{ (i,j) \colon R_{i,j} \subset R(0,n(p_c+\ep,t)) \}$. We prove that, as $n\to\infty$, the density of $\CC_0(t)$ in each of these boxes converges to the density of the infinite component in bond percolation with probability $\rho(x_{i,j},t)$ of an edge being open. 	

\begin{thm} \label{thm:inside}
For any $\ep,\delta>0$, as $t\to\infty$,
$$
P\left( \sup_{(i,j) \in \Lambda(t,\ep)} |D_{i,j}(t) - \theta(\rho(x_{i,j},t))| > \delta \right) \to 0.
$$
\end{thm}

\noindent
From this we get a result about the size of $\CC_0(t)$.

\begin{cor} \label{cor:inside}
$|\CC_0(t)|/t^{2/\alpha} \to \iint \theta( 1- \exp(-\|x\|^{-\alpha})) \, dx_2 \, dx_1$ as $t\to\infty$.
\end{cor}

Our proof of Theorem \ref{thm:inside} makes heavy use of the planar graph duality for two dimensional bond percolation. Consider bond percolation on the dual lattice
$\dual := \ZZ^2+(1/2,1/2)$ with nearest neighbor edges. Every edge $e$ on $\ZZ^2$ is paired with an edge $e^*$ on $\LL$ that has 
the same midpoint. If $e$ is open (resp.~closed), then $e^*$ is closed (resp.~open). The pairing means that if the density on the original
lattice is $p$, then the density on the dual lattice is $1-p$. We use $P^*_{1-p}$ to denote the percolation on the dual lattice.
It is known that there is a top-to-bottom open crossing
of $[a,b] \times [c,d]$ if an only if there is no left-to-right closed crossing of $[a-1/2,b+1/2] \times [c+1/2,d-1/2]$.
Having mentioned the exact size of the rectangles once,
we will ignore the 1/2's in what follows.  

Let $I_n = \left[-\lceil n/(C_1 \log n) \rceil -1, \lceil n/(C_1 \log n)\rceil\right]$ and for $j \in I_n$ let
\begin{align*}
R_j = [jC_1\log n, (j+1)C_1 \log n] \times [-n,n],\\ 
R^j = [-n,n] \times [jC_1\log n, (j+1)C_1 \log n].
\end{align*}
Note for the next step that the limits on $j$ are chosen so that
the first and last strips in each direction lie outside of $R(0,n(p_c+\ep,t))$. Let $\rad({\cal C}_x)$ be the radius of the cluster that contains $x$. It is known that in homogeneous percolation
\beq
P^*_{p_c-\ep}( \rad({\cal C}_x) \ge k ) \le Ce^{-\gamma_r k}, 
\eeq
for some constants $C$ and $\gamma_r$ that depend on $p_c - \ep$. So, if $n = n(p_c+\ep,t)$ and we pick $C_1$ large enough then
\beq
P^*_{p_c-\ep}( \rad({\cal C}_x) \ge C_1\log n ) \le n^{-3},
\label{dualrad}
\eeq
for all $x \in \LL$.
It follows from \eqref{dualrad} that, with high probability, for all $j \in I_n$: 
(i) there is no left to right dual crossing of any $R_j$ and hence each $R_j$ has an open top to bottom crossing;
and (ii) there is a left to right open crossing of all of the $R^j$.

\begin{figure}[h] 
  \centering
  \includegraphics[bb=0 0 360 335,height=2.0in,keepaspectratio]{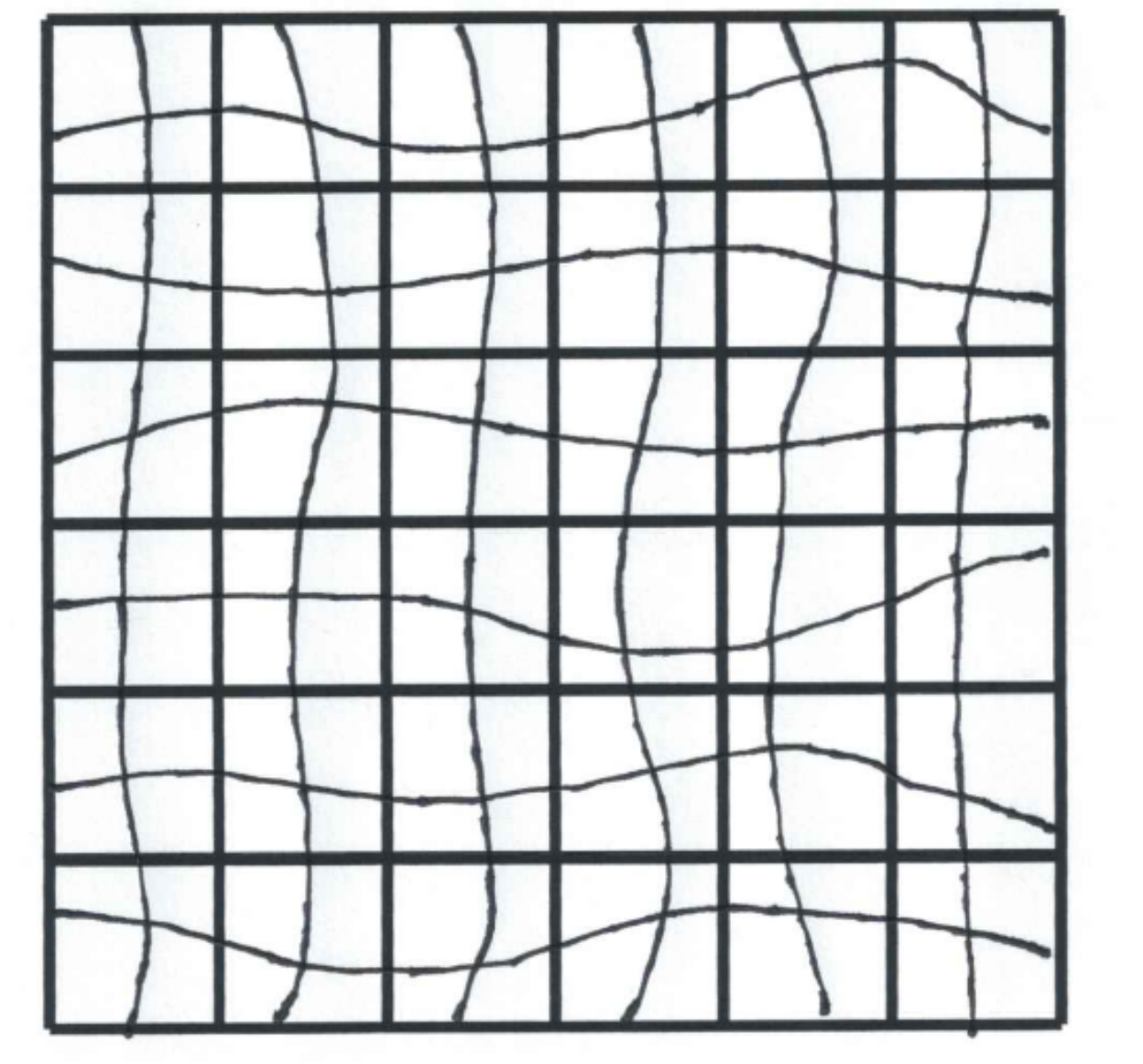}
  \caption{(i) and (ii) give us a net of interweaving crossings. }
  \label{fig:mesh2}
\end{figure}

Let $G(x,t)$ be the event that $\rad({\cal C}_x) > 2 C_1 \log n$.
It is easy to see that if $\|x-y\| > 4C_1 \log n$ then $G(x,t)$ and $G(y,t)$ are independent. Bounding the second moment of $|\CC_0(t) \cap R_{i,j}|$ and using Chebyshev's inequality in conjunction with a union bound over all of the boxes gives the desired result.\eopt

\medskip
After the results mentioned above were proved, we learned about {\it gradient percolation}. In 1985 Sapoval, Rosso, and Gouyet \cite{SRG} considered a model in which a site $(x,y)$ is occupied with probability
$$
p(y) = 1 - \frac{2}{\pi^{1/2}} \int_0^{y/(2t^{1/2})} e^{-u^2} \, du.
$$
This formula arose from a model in which particles do the simple exclusion process in the upper half-space and the $x$ axis is kept occupied by adding particles at empty sites. They looked at the geometry of the boundary of the connected component containing the $x$-axis, finding that the front was fractal with dimension $D_f = 1.76 \pm 0.002$. This paper has been cited 395 times according to Google Scholar. Proving rigorous result about the boundary was mentioned as an open problem in the survey Beffara and Sidorovicius \cite{BefSid} wrote for the Encyclopedia of Mathematical Physics, a five volume set first published in 2004 by Elsevier.

In 2008 Pierre Nolin \cite{Nolin} proved rigorous results for a related percolation model on the two dimensional honeycomb lattice. In the homogeneous version the plane is tiled with hexagons that are black with probability p and white with probability $1-p$. This is equivalent to site percolation on the triangular lattice. Since the pioneering work of Kesten \cite{K82} in the early 1980s, it has been known that the critical value for this model is 1/2. In 2001 Smirnov and Werner \cite{SmiWer} used conformal invariance and work of Kesten \cite{K87} on scaling relations to rigorously compute critical values for this model.

\begin{figure}[h]
\begin{center}
\begin{picture}(330,160)
\put(30,10){\line(1,2){60}}
\put(50,10){\line(1,2){60}}
\put(70,10){\line(1,2){60}}
\put(90,10){\line(1,2){60}}
\put(110,10){\line(1,2){60}}
\put(130,10){\line(1,2){60}}
\put(150,10){\line(1,2){60}}
\put(170,10){\line(1,2){60}}
\put(190,10){\line(1,2){60}}
\put(210,10){\line(1,2){60}}
\put(230,10){\line(1,2){60}}
\put(250,10){\line(1,2){60}}
\put(50,10){\line(-1,2){10}}
\put(70,10){\line(-1,2){20}}
\put(90,10){\line(-1,2){30}}
\put(110,10){\line(-1,2){40}}
\put(130,10){\line(-1,2){50}}
\put(150,10){\line(-1,2){60}}
\put(170,10){\line(-1,2){60}}
\put(190,10){\line(-1,2){60}}
\put(210,10){\line(-1,2){60}}
\put(230,10){\line(-1,2){60}}
\put(250,10){\line(-1,2){60}}
\put(210,130){\line(1,-2){50}}
\put(230,130){\line(1,-2){40}}
\put(250,130){\line(1,-2){30}}
\put(270,130){\line(1,-2){20}}
\put(290,130){\line(1,-2){10}}
\put(30,10){\line(1,0){220}}
\put(40,30){\line(1,0){220}}
\put(50,50){\line(1,0){220}}
\put(60,70){\line(1,0){220}}
\put(70,90){\line(1,0){220}}
\put(80,110){\line(1,0){220}}
\put(90,130){\line(1,0){220}}
\put(190,140){$\ell_N$}
\put(290,60){$N$}
\end{picture}
\caption{Nolin's parallelogram}
\end{center}
\end{figure}
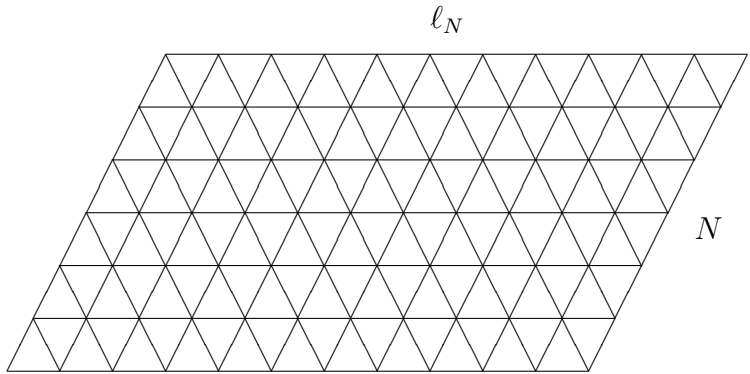

. 

Nolin considered percolation in a parallelogram with height $N$, length $\ell_N$, and interior angles of 60 and 120 degrees, with sites black with probability $1-y/N$ when $0 \le y \le N$. As in our result, the boundary of the cluster of black sites containing the $x$-axis will be close to the line $y=N/2$. Writing $\approx N^a$ for a quantity that is bounded below by $N^{a-\delta}$ and above by $N^{a+\delta}$ for any $\delta>0$, Nolin proved the following result, predicted in \cite{SRG}.

\begin{thm} \label{thm:nolin}
The boundary of the cluster containing the $x$-axis 
remains within $\approx N^{4/7}$ of the line $y = N/2$ and has length $\approx N^{3/7}\ell_N$.
\end{thm} 

\noindent
To connect with the original work in \cite{SRG}, Nolin says ``one can expect to observe a nontrivial limit, of fractal dimension 7/4, with an appropriate scaling (in $N^{4/7}$) of the axes, but the critical exponents obtained do not correspond to a fractal dimension of the limiting object.''

\begin{figure}[h] 
  \centering
  \includegraphics[width = 8 cm]{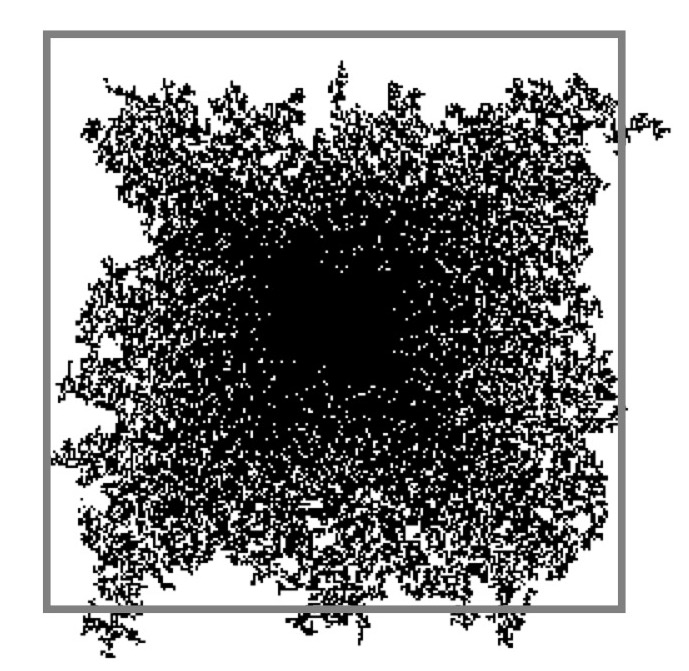}
  \caption{$\mathbb C_0$ for $\alpha=1$ when $n(p_c,t) = 150$ ($t = 104$). The ambient box has radius $150$.}
  \label{fig:sim}
\end{figure}

Since it is expected, but not yet proved, that the critical exponents are the same for bond percolation on the square lattice, we cannot convert Nolin's result into a theorem about our model. To make the connection between our result and his, let 
$$
N = n(p_c,t) = c_{p_c,\alpha} t^{1/\alpha},
$$
where $c_{p,\alpha}$ is defined in \eqref{cpa}. Changing variables 
$$
\rho((yN,0),t) = 1 - \exp(-t (c_{p_c,\alpha} yt^{1/\alpha})^{-\alpha}) = 1 - \exp( -  y^{-\alpha} \log 2 ) \equiv f(y).
$$
Near 1 we have $f(1+\delta) = 1/2 + f'(1)\delta + o(\delta)$. Theorem \ref{thm:outside} and \ref{thm:inside} imply that we can confine our attention this region. Only near the corners of the right-edge of $R(0,N)$ do we notice a difference between a model with probabilities that depend on $x$ and ours that depend on $\|x\|$, so it is reasonable to expect that the conclusion of Theorem \ref{thm:nolin} will hold for our model. Note that the formula for $f(y)$ tells us that boundary fluctuations will not depend on $\alpha$ but the density profile of $\CC_0(t)$ will. 

\section{Proof of Theorem \ref{thm:outside}}

\begin{proof}
Let $N = n(p_c-\ep,t)$. Using \eqref{dualrad} and the fact that $P$ and $P^*$ are the same (except for being defined on different lattices)
\beq
P_{p-\ep}(\rad(C_x) \ge C_1\log N ) \le N^{-3}.
\label{out1}
\eeq
Let $B_N$ be the event that there is an open path from $\partial R(0,N)$ to $\partial R(0, N + C_1\log N)$.
To bound $P(B_N)$ note that if there is such an open path then there is one 
that stays entirely in the annulus $R(0,N+C_1\log N) - R(0,N)$
where all of the bonds are open with probability $p_c-\ep$.
Using \eqref{out1} with a union bound gives $P(B_N) \le 8/N \to 0$.  
This implies
$$
P(\exists~x \in \CC_0(t): ||x|| \geq N + C_1\log N ) \to 0
$$
which proves the desired result.
\end{proof}

\section{Proof of Theorem \ref{thm:inside}}

We fix a time $t$, let $n = n(p_c + \epsilon, t)$ and partition the box $R(0, n)$ into two sets of strips $R_j$ and $R^j$, as described in Section \ref{Intro}. Define the following pair of events:
\begin{align*}
A_j  & = \{\text{$\exists$ a top-to-bottom crossing in $R_j = [jC_1\log n, (j+1)C_1 \log n] \times [-n,n]$}\},\\
A^j &= \{\text{$\exists$ a left-to-right crossing in $R^j = [-n,n] \times [jC_1\log n, (j+1)C_1 \log n]$}\}.
\end{align*}

\begin{lemma} \label{Ajbound}
For $j \in I_n$, (i)  $P(A_j)$, $P(A^j)$ $\geq$ $1 - n^2$ and
(ii) $P(\bigcap\limits_{j \in I_n} A_j \cap A^j) \geq 1- 2n^{-1}$.
\end{lemma}

\begin{proof}
By symmetry, it suffices to prove (i) for the events $A^j$. Denote the left and right edges of $R^j$ by $\partial^LR^j$ and by $\partial^RR^j$, respectively, and by using the dual lattice $\dual= \ZZ^2 + (1/2, 1/2)$ defined in Section \ref{Intro}, the complement of $A^j$
$$
A^{j,c} = \bigcup_{x \in \partial^LR^j}\text{\{$\exists$ an open path from $x$ to $\partial^RR^j$ in $\dual$}\}. 
$$
Using \eqref{dualrad} with a union bound, we have 
$$
P(A^{j,c}) \leq |\partial^LR^j|P(\rad(\mathcal{C}_{p_c-\epsilon}(x)) > C_1\log(n)) 
 \leq n \cdot n^{-3} = n^{-2}.
$$
proving our first claim.

To prove (ii), note that we have a total of $\le 2n$ horizontal and vertical strips and thus
$$
P\left(\bigcup_{j \in I_n} A_j^c \cup A^{j,c} \right) \leq 2n\cdot n^{-2}
 = 1 - 2n^{-1}.
$$
\end{proof}

Lemma \ref{Ajbound} guarantees that there exists a ``net" with mesh-size $C_1\log n$ throughout $R(0,n)$. It is necessary to show that $\mathbb C_0$ is captured by this net. 

\begin{lemma} \label{caught_in_the_net}
	$P(\text{there exists a closed edge in $[-C_1 \log n, C_1 \log n]^2$})  \to 0$.
\end{lemma}

\begin{proof}

 Let $R$ denote the square in the lemma statement. Since $t =c n^\alpha$ for some $c>0$, it follows that $ \max_{x \in R} (1- \rho(x,t)) \leq \exp( - c n^\alpha / C_1\log n).$
Using this estimate in a union bound over the $4C_1^2 \log^2 n$ edges in $R$ gives the claimed convergence.
\end{proof}

We now consider a second partition of our original box $R(0, n)$, by tiling it with boxes $R_{i,j} = [in^a,(i+1)n^a] \times [jn^a,(j+1)n^a]$, centered at $x_{i, j}$, as described in Section \ref{Intro}. We will argue that the density of open sites computed in each rectangle $R_{i, j}$, is, with high probability, close to the percolation probability when bonds are open with probability $\rho(x_{i,j},t)$. 

For points $x$ inside an arbitrary box $R_{i, j}$, we examine the behavior of $\theta(\rho(x, t))$, which is the percolation probability  probability measure for bond percolation with parameter $\rho(x, t)$. The following result shows that, as $t \to\infty$, $\theta(\rho(x, t))$ remains almost constant as $x$ varies within $R_{i, j}$.

\begin{lemma} \label{dens}
Let $n = n(p_c,t)$ and $a<1$. As $n\to\infty$
$$
\sup \{ |\theta(\rho(x, t)) - \theta(\rho(y, t))| : \|x-y\| \le 2n^a \} \to 0.
$$
 
\end{lemma}

\begin{proof} Since $p \to \theta(p)$ is uniformly continuous, and $\theta(p)=0$ for $p<p_c$, it suffices to show that 
$$
\sup \{ |\rho(x, t) - \rho(y, t)| : \|x-y\| \le n^a,  \|x\|, \|y\| \le n \} \to 0.
$$
A little algebra gives
$$
\rho(x, t) - \rho(y, t)  = e^{-t\|x\|^\alpha} ( 1- e^{-t[\|y\|^\alpha - \|x\|^{\alpha}]}).
$$ 
Suppose first that $\|x\| \le n^b$ where $a < b < 1$. The second term is $\le 1$. Since $n= c(p_c)t^{1/\alpha}$ the first is 
$\le \exp( - c t^{1-b} )\to 0$.

If $\|x\| \ge n^b$ and $\|x-y\| \le n^a$ then for large $n$, $\|y\| \ge n^b/2$. Let $u$ be the point in $\{x,y\}$ with smaller norm and let $v$ be the one with larger norm. Notice that
$$
\|u\|^{-\alpha} - \|v\|^{-\alpha}  \le \|u\|^{-\alpha} - (\|u\| +  2n^a )^{-\alpha} 
 = - \int_{\|u\|}^{\|u\|+2n^a} (-\alpha x^{-\alpha-1}) \, dx .
$$
So, we have
$\|u\|^{-\alpha} - \|v\|^{-\alpha} \le \alpha n^a (n^b/2)^{-(1+\alpha)} \to 0$
since $b>a$ and $\alpha>0$.
\end{proof}

\begin{lemma} \label{localdensity}
Let $\theta_{i, j} = \theta(\rho(x_{i, j}, t))$.
For each $\delta > 0$, there is a constant $C_2$, independent of $i,j\in I_n$ and of $\delta$, so that
 $$
P\left(| |\CC_0 \cap R_{i, j}| - \theta_{i,j}n^{2a}| > \delta n^{2a}\right) 
\leq \frac{C_2\log n^2}{\delta^2n^{2a}}.
$$
\end{lemma}

\begin{proof} To argue this, we define the following random variable
$$
S_{i, j} = \sum_{y \in R_{i, j}} 1\{\rad(\mathcal{C}_y \geq 2C_1 \log n)\},
$$
where $C_1\log n$ represents the lengths of the short sides of the rectangles $R_j$ and $R^j$. For all $y$ $\in$ $R_{i, j}$, let $A_y = \{\rad(\mathcal{C}_y \geq 2C_1 \log n)\}$. Recalling that this set of rectangles generates with high probability a net of open horizontal and vertical crossings, we note that $S_{i, j} = |\CC_0 \cap R_{i, j}|$. We now center the variable $S_{i, j}$ around its mean and define:
$$
\bar{S}_{i, j} = S_{i, j} - ES_{i, j}  = \sum_{y \in R_{i, j}} \left( 1\{A_y\} - \theta_{y}\right), 
$$
where $\theta_{y} = P(A_y)$ for all $y$. 
Since $E(\bar{S_k}) = 0$, we have
\begin{align*}
\var(\bar{S}_{i, j}) = E(\bar{S}_{i, j}^2) = E\left(\sum_{x, y \in R_{i, j}} \left(\mathds{1}\{A_y\}\mathds{1}\{A_x\} - \theta_{x}\theta_{y}\right)\right).
\end{align*} 
The random variables $\mathds{1}\{A_x\}$ and $\mathds{1}\{A_y\}$ are independent, if $|x-y| \geq 4C_1\log n$. 
Using this observation and the fact that $|E(\mathds{1}\{A_x\}\mathds{1}\{A_y\}) - \theta_{i, j}^2|$ $\leq$ 1, we obtain
\begin{align*}
E(\bar{S}_{i, j}^2) &= \sum_{|x-y| < 2C_1\log n}  \left( E(\mathds{1}\{A_x\}\mathds{1}\{A_y\}) - \theta_{i, j}^2 \right) \\
&\leq | \{(x, y) \in R_{i, j}: \|x-y\| < 4C_1\log n \}| \le C_2 n^{2a}{\log^2 n}.
\end{align*}
Using  Chebyshev's inequality gives
\begin{align*}
P(|\bar{S}_{i, j}| > \delta n^{2a}) \leq \frac{C_2 n^{2a}\log^2 n}{\delta^2(n^{2a})^2} =\frac{C_2 \log^2 n}{\delta^2n^{2a}}.
\end{align*}
Since Lemma \ref{dens} implies 
$$
|R_{i,j}|^{-1} \sum_{y\in R_{ij}} \theta_y - \theta_{i,j} \to 0
$$
this proves the lemma.
\end{proof}

Given this series of results, we can now easily complete the

\begin{proof}[ Proof of Theorem \ref{thm:inside}]
By Lemma \ref{caught_in_the_net} $\mathbb C_0$ connects to the ``net" from Lemma \ref{Ajbound}. Thus, it contains a crossing of every strip $R^j$ and $R_j$ for $j \in I_n$. 
Next,  note that for each $\delta > 0$
$$
P\left( \sup_{(i,j) \in \Lambda(t,\ep)} |D_{i,j}(t) - \theta(P(x_{i,j},t))| > \delta \right)  \leq \sum_{(i,j) \in \Lambda(t,\ep)}P(|S_{i, j} - \theta_{i, j} n^{2a}| > \delta n^{2a}).
$$
Using Lemma \ref{localdensity}  the above is
$$
\leq n^{2-2a}\frac{C_2\log^2 n}{\delta^2n^{2a}} \leq n^{2 -4a}\frac{C_2\log^2 n}{\delta^2} \to 0,
$$
since $a > \frac{1}{2}$. 
\end{proof}

\begin{proof}[Proof of Corollary \ref{cor:inside}] Observe that
$|\CC_0(t)| = \sum_{i,j\in I_n} |\CC_0(t) \cap R_{i,j}|$. Theorem \ref{thm:inside} implies 
$$
\frac{1}{n^2} \left| \sum_{i,j\in I_n} |\CC_0(t) \cap R_{i,j}| - \sum_{i,j} \theta_{i,j} \right| \to 0.
$$ 
Scaling space by $t^{1/\alpha} = O(n)$ and noting that the squares now have side length $O(n^{a-1})$,
we have 
$$
\frac{1}{t^{2/\alpha}} \sum_{i,j} \theta_{i,j} \to \iint (1 - \exp(-\|x\|^{-\alpha}) \, dx_2 \, dx_1,
$$
which completes the proof.
\end{proof}

\end{document}